\documentclass[11pt]{amsart}

\usepackage{hyperref}
\usepackage{ucs}
\usepackage[utf8]{inputenc}
\usepackage{amsmath}
\usepackage{amssymb}
\usepackage{amsthm}
\usepackage[english]{babel}
\usepackage{fontenc}
\usepackage{graphicx, tikz}
\usepackage{tikz-cd}
\usetikzlibrary{arrows}
\usepackage{amscd}

\author{Johannes Schmitt}
\title{Dimension theory of the moduli space of twisted $k$-differentials}
\address{Departement Mathematik, ETH Z\"urich, R\"amistrasse 101, 8092 Z\"urich, Switzerland}
\email{johannes.schmitt@math.ethz.ch}
\date{\today}
\renewcommand{\dim}{\text{dim}}

\newcommand{\todo}[1]{}
\newcommand{\todoOld}[1]{}
\newcommand{\todoAlt}[1]{}
\newcommand{\todoFin}[1]{}

\newcommand{\diviso}{\text{div}}

\newcommand{\comment}[1]{}

\newcommand{\detex}[1]{}  %detailed explanations

\begin{document}
\begin{abstract}
 In this note we extend the dimension theory for the spaces $\widetilde{\mathcal H}_g^k(\mu)$ of twisted $k$-differentials defined by Farkas and Pandharipande in \cite{fp} to the case $k>1$. In particular, we show that the intersection $\widetilde{\mathcal H}_g^k(\mu) \cap \mathcal{M}_{g,n}$ is a union of smooth components of the expected dimensions for all $k\geq 0$. We also extend a conjectural formula from \cite{fp} for a weighted fundamental class of $\widetilde{\mathcal H}_g^k(\mu)$ and provide evidence in low genus.
\end{abstract}

 \maketitle

 \newtheoremstyle{test}% name
  {}%      Space above, empty = `usual value'
  {}%      Space below
  {\it}% Body font
  {}%         Indent amount (empty = no indent, \parindent = para indent)
  {\bfseries}% Thm head font
  {.}%        Punctuation after thm head 
  { }% Space after thm head: \newline = linebreak
  {}%         Thm head spec
 
 \theoremstyle{test}
\newtheorem{Def}{Definition}[section]
\newtheorem{Exa}[Def]{Example}
\newtheorem{Rmk}[Def]{Remark}
\newtheorem{Exe}[Def]{Exercise}
\newtheorem{Theo}[Def]{Theorem}
\newtheorem{Lem}[Def]{Lemma}
\newtheorem{Cor}[Def]{Corollary}
\newtheorem{Pro}[Def]{Proposition}
\newtheorem*{Rmk*}{Remark}   %no numbering
\newtheorem*{Exa*}{Example}   %no numbering
\newtheorem*{Pro*}{Proposition} %no numbering
\newtheorem*{Def*}{Definition}
\newtheorem*{Cor*}{Corollary}
\newtheorem*{Lem*}{Lemma}
\newtheorem*{Theo*}{Theorem}

\section{Introduction}
In \cite{fp}, Farkas and Pandharipande define the moduli space $\widetilde{\mathcal H}_{g}^k(\mu)$ of twisted $k$-differentials associated to an integer partition $\mu=(m_1, \ldots, m_n)$ of $k(2g-2)$. It is a closed substack $\widetilde{\mathcal H}_{g}^k(\mu)\subset \overline{\mathcal{M}}_{g,n}$, whose interior points ${\mathcal H}_{g}^k(\mu) = \widetilde{\mathcal H}_{g}^k(\mu) \cap {\mathcal{M}}_{g,n}$ parametrize curves $(C,p_1, \ldots, p_n)$ such that there exists a meromorphic section $s_0$ of $\omega_C^{\otimes k}$ with zeroes and poles at $p_1, \ldots, p_n$ with orders prescribed by $\mu$. That is $\text{div}(s_0) = \sum_{i=1}^n m_i p_i$ or equivalently
\[\mathcal{O}_C\left(\sum_{i=1}^n m_i p_i\right)= \omega_C^{\otimes k}.\]
On the boundary of $\overline{\mathcal{M}}_{g,n}$, this definition is modified by allowing to introduce twists $I(q,D')=-I(q,D'') \in \mathbb{Z}$ at nodes $q$ of $C$ where two distinct components $D',D''$ of $C$ meet. These twists need to satisfy certain combinatorial conditions. Let $\nu : C_I \to C$ be the partial normalization of $C$ resolving the collection $N_I$ of all nodes $q$  of $C$ with nonzero twist $I(q,D_q')$ and let $q', q''$ be the preimages of $q$ contained in the preimages of $D_q',D_q''$. Then $(C,p_1, \ldots, p_n)$ is called \emph{$k$-twisted canonical} and is contained in $\widetilde{\mathcal H}_{g}^k(\mu)$ iff we have an equality of line bundles
\[\nu^* \mathcal{O}_C\left(\sum_{i=1}^n m_i p_i\right) = \nu^*\omega_C^{\otimes k} \otimes \mathcal{O}_{C_I}\left( \sum_{q \in N_I} I(q,D_q') q' + I(q,D_q'') q''\right)\]
on $C_I$. For details see \cite[Section 0 and Definition 20]{fp}. See also \cite{guere} for an approach to $k$-twisted canonical divisors via log geometry.

For $k=1$, the dimension theory of $\widetilde{\mathcal H}_{g}^k(\mu)$ has been studied in \cite{fp}. 
\begin{itemize}
 \item In the holomorphic case, i.e. if all $m_i \geq 0$, the components of  the closure $\overline{\mathcal H}_{g}^1(\mu)$ of ${\mathcal H}_{g}^1(\mu)$ have dimension $2g-2+n$ and all other components of $\widetilde{\mathcal H}_{g}^1(\mu)$ have dimension $2g-3+n$. 
 \item In the strictly meromorphic case, i.e. if there exists some $m_i<0$, all components of $\widetilde{\mathcal H}_{g}^1(\mu)$ have dimension $2g-3+n$. 
\end{itemize}
In the case $k=0$, the space $\widetilde{\mathcal H}_{g}^k(\mu)$ has components of various dimensions, all at least $2g-3+n$.

Extending the result for  $k=1$ above, we prove the following. 

\begin{Theo} \label{Theo:dim}
 For $k \geq 1$, all components of $\widetilde{\mathcal{H}}_g^k(\mu)$ are of dimension $2g-3+n$, except in the case that all parts of $\mu$ are nonnegative and divisible by $k$. Then the sublocus 
 \[\mathcal{\overline  H}_g^1\left(\frac{1}{k} \mu \right) \subset \widetilde{\mathcal H}_g^k(\mu)\]
 is a union of irreducible components of dimension $2g-2+n$ and all other components of $\widetilde{\mathcal H}_g^k(\mu)$ have dimension $2g-3+n$.
\end{Theo}

This was conjectured in \cite{fp}. We are going to see that the crucial point in the argument is the dimension theory of the open locus ${\mathcal{H}}_g^k(\mu)$. 

\begin{Pro} \label{Pro:dimTC}
 Let $k \geq 0$, $(C,p) \in \mathcal{H}_g^k(\mu)$ and let $s_0$ be a meromorphic section of $\omega_C^{\otimes k}$ with $\diviso(s_0)=\sum_{i=1}^n m_i p_i$. Consider the condition
 \begin{equation} \label{eqn:condhol}
  \parbox{0.8\linewidth}{all $m_i$ are nonnegative and divisible by $k$ and there exists a meromorphic section $\tilde s_0$ of $\omega_C$ such that $s_0 = \tilde s_0^{\otimes k}$.}
   \tag{*}
 \end{equation}
 Then
 \[\dim\ T_{(C,p)} \mathcal{H}_g^k(\mu) = 2g-2+n,\]
 if (\ref{eqn:condhol}) is satisfied and 
 \[\dim\ T_{(C,p)} \mathcal{H}_g^k(\mu) = 2g-3+n\]
 otherwise. Here we require that if $k=0$, the partition $\mu$ should not be the trivial partition $\mu=(0, \ldots, 0)$ (otherwise $\mathcal{H}_g^0(\mu)=\mathcal{M}_{g,n}$).
 
 In particular, all components of $\mathcal{H}_g^k(\mu)$ are smooth of the dimensions determined above.
\end{Pro}
This result has been known in parts of the community, but to our knowledge it has not been published. 
% We are for instance aware of a proof by G. Mondello, which has a similar overall strategy but uses different techniques in various steps compared to our proof (G. Mondello, personal communication with R. Pandharipande, November 2015).
We are for instance aware of a proof by G. Mondello (G. Mondello, personal communication with R. Pandharipande, November 2015). Similar to our proof below, it first uses deformation theory to reduce the claim to the computation of the cokernel of a map of first cohomology groups of sheaves on $C$ induced by a sheaf map. While Mondello uses Serre duality to reduce to a computation in $H^0$, we analyze kernel and cokernel of the map of sheaves directly and then use long exact sequences. The deformation theoretic part of our argument below follows closely the proof of Polishchuk in \cite{polishchuk}, who proved the result for $k=1$ and $\mu$ holomorphic. \todo{Verify date and overall appropriateness of acknowledgement}

In \cite{bcggm} the question if a given twisted differential $(C,p_1, \ldots, p_n) \in \widetilde{\mathcal H}_{g}^k(\mu)$ is actually contained in the closure $\overline{\mathcal H}_{g}^k(\mu)$ is answered in terms of orders and residue conditions on the dual graph of $C$ in the case $k=1$. The question for $k>1$ is still open (though a solution is announced in \cite{bcggm}) and we do not address it in the following.

In Section \ref{Sect:conjecture} we present a generalized version of a conjecture from \cite{fp}, which  relates a fundamental class of $\widetilde{\mathcal H}_{g}^k(\mu)$ with explicit weights on the components inside the boundary to a tautological class described by Pixton in \cite{pixtondr}. 
Here we must distinguish cases.
\begin{itemize}
 \item If $\mu$ has an entry which is negative or not divisible by $k$, all components of $\widetilde{\mathcal H}_{g}^k(\mu)$ are of codimension $g$ and we obtain a straightforward generalization of the conjectural formula in \cite{fp} (see Conjecture A).
 \item If $\mu=k \mu'$ for $\mu'$ nonnegative, we replace the codimension $g-1$ contribution $\overline{\mathcal H}_{g}^1(\mu')$ of this formula by a codimension $g$ virtual fundamental class $[\overline{\mathcal H}_{g}^1(\mu')]^{\text{vir}}$. This virtual cycle is obtained by inserting the parameters $k=1$ and $\mu'$ in the formula of Conjecture A  (see Conjecture A$'$).
\end{itemize}
We check both conjectures for genera $g=0,1$ and for each conjecture two nontrivial cases in genus $g=2$. We also show how for all $k \geq 1$  the class $[\overline{\mathcal{H}}_g^k(\mu)]$ of the closure of $\mathcal{H}_g^k(\mu)$ is determined by Conjecture A in the case where $\mathcal{H}_g^k(\mu)$ has pure codimension $g$, i.e. for $\mu \neq k \mu'$ with $\mu'$ nonnegative.

In the appendix of the paper, we give an elementary argument, explained to us by Dimitri Zvonkine, showing that an untwisted node of a curve $(C,p) \in \widetilde{\mathcal H}_{g}^k(\mu)$ can be smoothed inside $\widetilde{\mathcal H}_{g}^k(\mu)$. This technical result is needed to identify the boundary components of $\widetilde{\mathcal H}_{g}^k(\mu)$ in terms of their dual graphs.

\section*{Acknowledgements}
I would like to thank my advisor Rahul Pandharipande for introducing me to the topic of twisted differentials and for his guidance throughout the writing of this paper. I am also grateful to  Dimitri Zvonkine for explaining the proof in the appendix and to Aaron Pixton for verifying the tautological relations appearing in the examples in Section \ref{Sect:ExaGen2}. Finally, I want to thank Felix Janda for very helpful discussions and advice and for his comments on a preliminary version of the paper.

I am supported by the grant SNF-200020162928.

\section{Dimension estimates}
Let $g,n,k \geq 0$ and assume $2g-2+n >0$. Let $\mu = (m_1, \ldots, m_n) \in \mathbb{Z}^n$ be a partition of $k(2g-2)$. 

For the proof of Theorem \ref{Theo:dim}, many of the dimension estimates for the case $k=1$ in \cite{fp} carry over verbatim to the case of general $k$. By \cite[Theorem 21]{fp} we know that all components of $\widetilde{\mathcal H}_g^k(\mu)$ have dimension at least $2g-3+n$. 

Now assuming Proposition \ref{Pro:dimTC}, we know that the dimension of $\mathcal{H}_g^k(\mu)$ is bounded above by $2g-2+n$ and in the meromorphic case it is even bounded by $2g-3+n$. Using this, the proof of \cite[Proposition 7]{fp} shows the following.

\begin{Cor} \label{Cor:bdrycompdim}
 For $k \geq 1$, every irreducible component $Z$ of $\widetilde{\mathcal H}_g^k(\mu)$ supported in the boundary has dimension at most $2g-3+n$. Here equality can be achieved only if the dual graph of a generic element $(C,p) \in Z$ is a star-graph (see Section \ref{Sect:Hgkmu}).
 % Here, the twist must be a positive multiple of k, all markings going to non-central vertices must be nonnegative and divisible by k and we basically must be in the product of the moduli space of the central vertex with the spaces H_{g_i}^1(1/k \mu_i)
\end{Cor}

\begin{Rmk}
 Note that for $k=0$, the proof of \cite[Proposition 7]{fp} breaks down because for $\mu=(0, \ldots, 0)$ the codimension of $\mathcal{H}_g^0(\mu)$ is $0$ and not $g$.
\end{Rmk}
We see that Corollary \ref{Cor:bdrycompdim} covers all components of $\widetilde{\mathcal H}_g^k(\mu)$ supported in the boundary and Proposition \ref{Pro:dimTC} treats the components in the interior $\mathcal{M}_{g,n}$. Here we note that the $(C,p) \in \mathcal{H}_g^1(\frac{1}{k} \mu)$ are exactly the points of $\mathcal{H}_g^k(\mu)$ such that the corresponding meromorphic section $s_0$ has a $k$-th root in $\omega_C$. This finishes the proof of Theorem \ref{Theo:dim}, assuming Proposition \ref{Pro:dimTC}.

To show the Proposition, we adapt the original dimension estimate from \cite{polishchuk} by Polishchuk. Let $\mathcal{C} \to \mathcal{M}_{g,n}$ be the universal curve over $\mathcal{M}_{g,n}$ and let $\mathcal{J}^d$ be the relative Jacobian of degree $d$ over $\mathcal{M}_{g,n}$ for $d \in \mathbb{Z}$. Then $\mathcal{H}_g^k(\mu)$ can be defined as a fibre product involving the following morphisms:
\begin{itemize}
 \item $\sigma_k^\mu : \mathcal{M}_{g,n} \to \mathcal{J}^{k(2g-2)}$ sending $(C,p)$ to  $(C,p,\mathcal{O}_C(\sum_{i=1}^n m_i p_i))$,
 \item $c : \mathcal{M}_{g,n} \to \mathcal{J}^{2g-2}$ sending $(C,p)$ to $(C,p,\omega_C)$,
 \item $\psi_k : \mathcal{J}^{2g-2} \to \mathcal{J}^{k(2g-2)}$ sending $(C,p,L)$ to $(C,p,L^{\otimes k})$.
\end{itemize}
Then indeed the diagram
\begin{equation}
 \begin{tikzcd}
\mathcal{H}_g^k(\mu) \arrow{rr} \arrow{d} & & \mathcal{M}_{g,n} \arrow{d}{\sigma_k^\mu}\\
\mathcal{M}_{g,n} \arrow{r}{c} & \mathcal{J}^{2g-2} \arrow{r}{\psi_k} & \mathcal{J}^{k(2g-2)}
\end{tikzcd} 
\end{equation}
is cartesian. 

Let $(C,p)\in \mathcal{H}_g^k(\mu)$ and let $s_0$ be a meromorphic section of $\omega_C^{\otimes k}$ with $\diviso(s_0)=\sum_{i=1}^n m_i p_i$. The cartesian diagram above implies that the tangent space of $\mathcal{H}_g^k(\mu)$ at $(C,p)$ is the kernel of the map 
\begin{equation} \label{eqn:tangspace} T_{(C,p)} \mathcal{M}_{g,n} \oplus T_{(C,p)} \mathcal{M}_{g,n} \xrightarrow{(d \sigma_k^\mu, -d(\psi_k \circ c))} T_{(C,p,\omega_C^{\otimes k})} \mathcal{J}^{k(2g-2)}.\end{equation}
Let $d$ be the dimension of the cokernel of this map, then 
\begin{align*}
 \dim\ T_{(C,p)} \mathcal{H}_g^k(\mu) &= 2(3g-3+n)-(3g-3+n+g)+d\\ &= 2g-3+n+d.
\end{align*}
Thus we need to show that $d=1$ if (\ref{eqn:condhol}) holds and $d=0$ otherwise.

As in Polishchuk's paper, we can explicitly identify the above tangent spaces and the morphisms between them. We have
\begin{align*}
 T_{(C,p)} \mathcal{M}_{g,n} &= H^1(C,\mathcal{T}_C(-p_1 - \ldots - p_n)),\\
 T_{(C,p, L)} \mathcal{J}^d &= H^1(C,A_{L,p}),
\end{align*}
for $d \in \mathbb{Z}$ and $L$ a line bundle on $C$ of degree $d$. Here $A_{L,p}$ is the sheaf of differential operators $L \to L$ of order $\leq 1$ with vanishing symbol at $p_1, \ldots, p_n$. In the following, let $E=p_1 + \ldots + p_n$.
\begin{Lem} \label{Lem:Aisom}
 The tangent map of the morphism
 \[\mathcal{J}^0 \to \mathcal{J}^{k(2g-2)},\quad (C,p,L) \mapsto \left(C,p,L\left(\sum_i m_i p_i\right)\right) \]
 at $(C,p,\mathcal{O}_C)$ is the map $H^1(C, A_{\mathcal{O}_C,p}) \to H^1(C,A_{\omega_C^{\otimes k},p})$ induced by the isomorphism of sheaves
 \[A_{\mathcal{O}_C,p} \cong A_{\omega_C^{\otimes k},p}\]
 sending an operator $\partial' : \mathcal{O}_C \to \mathcal{O}_C$ to the operator
 \[\partial: \omega_C^{\otimes k} \to \omega_C^{\otimes k},\quad  s \mapsto s_0 \partial'\left(\frac{s}{s_0}\right).\]
\end{Lem}
\begin{proof}
 See the proof of \cite[Lemma 2.2]{polishchuk}. We note that for a pole $p_i$ of $s_0$ of order $|m_i|$ (and $s$ regular around $p_i$), we have $\text{ord}_{p_i} \partial'(\frac{s}{s_0}) \geq  |m_i|$ as the symbol of $\partial'$ vanishes at $p_i$. Thus the expression $s_0 \partial'(\frac{s}{s_0})$ is again regular around $p_i$. 
 \todoOld{We need that the induced isomorphism in $H^1$ is the tangent map of the morphism $\mathcal{J}^0 \to \mathcal{J}^{k(2g-2)}, (C,L) \mapsto (C,L(\sum_i m_i p_i))$. Note: then isomorphism is basically clear.}
\end{proof}

\begin{Lem} \label{Lem:Tpsik}
 For $L$ a line bundle on $C$ of degree $d$, the tangent map 
 \[d \psi_k : H^1(C,A_{L,p}) \to H^1(C,A_{L^{\otimes k},p})\]
 of $\psi_k : \mathcal{J}^d \to \mathcal{J}^{dk}$ at $(C,p,L)$ is induced by the sheaf map $\Psi: A_{L,p} \to A_{L^{\otimes k},p}$ sending an operator $\partial: L \to L$ to 
 \begin{align*}
  \partial \otimes \text{id}^{\otimes k-1} + \text{id} \otimes \partial \otimes \text{id}^{\otimes k-2} + \ldots + \text{id}^{\otimes k-1} \otimes \partial : L^{\otimes k} \to L^{\otimes k}.
 \end{align*}
\end{Lem}
\begin{proof}
Assume we are given a first order deformation
\[\begin{tikzcd}
   C\arrow{d} \arrow{r} &\mathcal{C}\arrow{d}\\
   \text{Spec}(\mathbb{C}) \arrow{r} & \text{Spec}(\mathbb{C}[\epsilon]/(\epsilon^2))
  \end{tikzcd}
\]
of $(C,p)$ and a line bundle $\mathcal{L}$ on $\mathcal{C}$ deforming $L$. Then for an affine cover $\mathcal{U} = (U_\alpha)_{\alpha \in A}$ of $C$ trivializing $\mathcal{L}$ and $U_{\alpha \beta} = U_\alpha \cap U_\beta$ we have a cycle $(d_{\alpha \beta} \in \mathcal{Z}^1(U_{\alpha \beta},\mathcal{T}_C(-E)))$ describing the deformation of $(C,p)$. If $f_{\alpha \beta}$ are the transition functions of $L$, the transition functions of $\mathcal{L}$ have the form $F_{\alpha \beta} = f_{\alpha \beta} + \epsilon g_{\alpha \beta}$. Then under the identification of first-order deformations of $(C,p,L)$ with $H^1(C,A_{L,p})$ this data corresponds to the $1$-cycle
\[\left(\partial_{\alpha \beta} = d_{\alpha \beta} + \frac{g_{\alpha \beta}}{f_{\alpha \beta}} \in A_{L,p}(U_{\alpha \beta}) \right). \]
Under the map $\psi_k$, the data $d_{\alpha \beta}$ of the deformation of $(C,p)$ remains unchanged, but the transition functions of $\mathcal{L}^{\otimes k}$ are now
\[F_{\alpha \beta}^k = f_{\alpha \beta}^k + k \epsilon g_{\alpha \beta} f_{\alpha \beta}^{k-1}.\]
On the other hand, the sheaf map $\Psi$ above gives us a $1$-cycle $(\Psi(\partial_{\alpha \beta}) \in A_{L^{\otimes k},p}(U_{\alpha \beta}))$ and for a section $s \otimes 1 \otimes \ldots \otimes 1$ of $L^{\otimes k}$ on $U_{\alpha \beta}$ we have
\begin{align*}
 &\Psi(\partial_{\alpha \beta}) s \otimes 1 \otimes \ldots \otimes 1\\
 =& (d_{\alpha \beta} s + \frac{g_{\alpha \beta}}{f_{\alpha \beta}} s) \otimes 1 \otimes \ldots \otimes 1 + s \otimes (\underbrace{d_{\alpha \beta} 1}_{=0} +\frac{g_{\alpha \beta}}{f_{\alpha \beta}} 1) \otimes 1 \otimes \ldots \otimes 1 + \ldots \\
 =&(d_{\alpha \beta} s + s \frac{k g_{\alpha \beta}}{f_{\alpha \beta}}) \otimes 1 \otimes \ldots \otimes 1.
\end{align*}
Indeed, this data corresponds to the transition functions
\[f_{\alpha \beta}^k + \epsilon f_{\alpha \beta}^k \frac{k g_{\alpha \beta}}{f_{\alpha \beta}} = F_{\alpha \beta}^k\]
as claimed.
%  The sheaf $A_{L,p}$ sits inside an exact sequence
%  \[0 \to \mathcal{O}_C \to A_{L,p} \to \mathcal{T}_C(-E) \to 0\]
%  and the map $\Psi$ extends to a sequence of maps
%  \[\begin{tikzcd}
%     0 \arrow{r} &\mathcal{O}_C \arrow{r} \arrow{d}{m_k} &A_{L,p} \arrow{r} \arrow{d}{\Psi} &\mathcal{T}_C(-E) \arrow{r} \arrow{d}{\text{id}}&0\\
%     0 \arrow{r} &\mathcal{O}_C \arrow{r} &A_{L^{\otimes k},p} \arrow{r}  &\mathcal{T}_C(-E) \arrow{r}&0    
%    \end{tikzcd}
% \]
%  where $m_k$ denotes multiplication with $k$. Look at the induced diagram in cohomology.
%  \[\begin{tikzcd}
%    \ldots \arrow{r} &H^1(\mathcal{O}_C) \arrow{r} \arrow{d}{\cdot k} &H^1(A_{L,p}) \arrow{r} \arrow{d}{H^1(\Psi)} & H^1(\mathcal{T}_C(-E)) \arrow{r} \arrow{d}{\text{id}}&0\\
%     \ldots \arrow{r} &H^1(\mathcal{O}_C) \arrow{r} &H^1(A_{L^{\otimes k},p}) \arrow{r}  &H^1(\mathcal{T}_C(-E)) \arrow{r}&0    
%    \end{tikzcd}
% \]
%  To show that $d \psi_k = \Psi$ note that at $(C,p,L)$, the map $\psi_k$ leaves the underlying curve $(C,p)$ fixed. Thus it induces the identity on the first order deformation space $H^1(\mathcal{T}_C(-E)$ of $(C,p)$. On the other hand, for $(C,p)$ fixed, the tangent map of $L \mapsto L^{\otimes k}$ in the Jacobian of $C$ is multiplication by $k$, just as for the map $H^1(\Psi)$ above. This proves that the maps agree. \todo{Not complete!}
\end{proof}

\begin{Cor} \label{Cor:Tpsik}
 The tangent map
 \[d(\psi_k \circ c) : H^1(C,\mathcal{T}_C(-E)) \to H^1(C,A_{\omega_C^{\otimes k},p})\]
 to the morphism $\psi_k \circ c : \mathcal{M}_{g,n} \to \mathcal{J}^{k(2g-2)}$ at $(C,p)$ is induced by the map 
 \[\mathcal{T}_C(- E) \to A_{\omega_C^{\otimes k},p},\quad v \mapsto L_v,\]
 where $L_v$ is the Lie-derivative along the tangent field $v$.
\end{Cor}
\begin{proof}
 The case $k=1$ (and thus $\psi_k = \text{id}$) was shown in \cite[Lemma 2.3]{polishchuk}. For general $k$ we know $d(\psi_k \circ c) = d \psi_k \circ d c$ and $d c$ is induced by $v \mapsto (L_v: \omega_C \to \omega_C)$. Thus by Lemma \ref{Lem:Tpsik}, $d(\psi_k \circ c)$ is induced by
 \[v \mapsto L_v \otimes \text{id}^{\otimes k-1} +  \ldots + \text{id}^{\otimes k-1} \otimes L_v = L_v : \omega^{\otimes k} \to \omega^{\otimes k }. \]
 Here we use that $L_v (S \otimes T) = (L_v S) \otimes T + S \otimes (L_v T)$ for tensor fields $S,T$.
\end{proof}
\begin{Lem}
 The cokernel of the map (\ref{eqn:tangspace}) is isomorphic to the cokernel of the map 
 \[H^1(C,\mathcal{T}_C(-E)) \xrightarrow{H^1(\varphi)} H^1(C,\mathcal{O}_C)\]
 induced by the sheaf map 
 \[\varphi: \mathcal{T}_C(-E) \to \mathcal{O}_C,\quad v \mapsto \frac{1}{s_0} L_v(s_0).\] 
\end{Lem}
\begin{proof}
 This can be seen exactly as in the proof of \cite[Proposition 2.4]{polishchuk}, where we use Lemma \ref{Lem:Aisom}, Lemma \ref{Lem:Tpsik} and Corollary \ref{Cor:Tpsik} instead of the corresponding results in \cite{polishchuk}. 
%  the map $d \sigma_k^\mu$ is equal to the composition
%  \[H^1(\mathcal{T}_C(-E)) \xrightarrow{\alpha'} H^1(A_{\mathcal{O}_C,p}) \xrightarrow{\sim} H^1(A_{\omega_C^{\otimes k},p}).\]
%  Here $\alpha'$ is the tangent map of the morphism $\mathcal{M}_{g,n} \to \mathcal{J}^0, (C,p) \mapsto (C,p, \mathcal{O}_C$ and the second map comes from the isomorphism in Lemma \ref{Lem:Aisom}.
\end{proof}
To prove Proposition \ref{Pro:dimTC} we must show that the cokernel of $H^1(\varphi)$ has dimension $1$ if (\ref{eqn:condhol}) is satisfied and dimension $0$ otherwise.

Our strategy for computing the cokernel above is to explicitly identify the kernel $\mathcal{K}$ and the cokernel $\mathcal{N}$ of the sheaf map $\varphi : \mathcal{T}_C(-E) \to \mathcal{O}_C$ (in the analytic category) and then to use long exact sequences in cohomology induced by the exact sequence 
\[0 \to \mathcal{K} \to \mathcal{T}_C(-E) \xrightarrow{\varphi} \mathcal{O}_C \to \mathcal{N} \to 0.\]
We will first treat the case $k \geq 1$.
\begin{Pro} \label{Pro:kercoker}
 Let $k \geq 1$. The subsheaf $\mathcal{K}$ of $\mathcal{T}_C(-E)$ associates to an open set $U \subset C$ the vector fields $v$ on $U$ vanishing at all $p_i$ such that the natural pairing $\langle v^{\otimes k}, s_0 \rangle$ induced from the contraction $\mathcal{T}_C^{\otimes k} \otimes (\mathcal{T}_C^{\vee})^{\otimes k} \to \mathbb{C}$ is locally constant. The sheaf $\mathcal{K}$ can be expressed as $\mathcal{K} = \iota_! \widetilde{\mathcal K}$, where 
 \[\iota : \tilde C= C \setminus \{p_i: m_i \geq 0\text{ or $k$ does not divide $m_i$}\} \hookrightarrow C\]
 is the inclusion of the complement some of the $p_i$. The sheaf $\widetilde{\mathcal K}$ is a local system of rank $1$ on $\tilde C$. 
 
 The cokernel $\mathcal{N}$ of $\varphi$ is isomorphic to $j_* \underline{\mathbb{C}}$, where 
 \[j : P=\{p_i: m_i<0\text{ and $k$ divides $m_i$}\} \hookrightarrow C.\]
%  Let $p_i$ be a marking with $m_i<0$ divisible by $k$ and $U$ a neighbourhood of $p_i$ containing none of the other $p_j$. Assume $U$ can be identified via a fixed biholomorphic map to some open subset of $\mathbb{C}$. Then the induced map $\mathcal{O}_C(U) \to \mathcal{N}(U) \cong \mathbb{C}$ sends a function $h$ on $U$, given around $p_i$ by $h(z) = \sum_{j=0}^\infty c_j z^j$, to a fixed linear combination of the coefficients $c_0, \ldots, c_{m_i/k-1}$ where $c_{m_i/k-1}$ has nonzero coefficient. 
\end{Pro}
\begin{proof}
 Fix an open set $U \subset C$ with local coordinate $z$ and write $s_0 = g (dz)^k$, where $g$ is a meromorphic function on $U$. Given a vector field $v=f \frac{d}{dz}$ on $U$, we have
 \begin{align*}
   \varphi(U)(v) &=  \frac{1}{s_0} L_v(s_0) = \frac{1}{g (dz)^k} L_v(g (dz)^k) \\
   &= \frac{f \frac{dg}{dz} + k g \frac{df}{dz}}{g}= f \frac{d\log(g)}{dz} + k \frac{df}{dz}.
 \end{align*}
 This is zero iff $\frac{d\log(g)}{dz}+ k \frac{d\log(f)}{dz} = 0$, i.e. $f^k g= \langle v^{\otimes k}, s_0 \rangle$ is constant. If this has a nonzero solution $f$ then all solutions are scalar multiples of $f$ (as $g$ only has isolated zeroes).
 
 Concerning the (local) existence of a solution around a point $p$, note that if $g$ has an isolated zero at $p$, this equation has no nonzero holomorphic solution $f$ and such zeroes $p$ of $g$ occur exactly on the $p_i$ with $m_i >0$. If $g$ is regular but nonzero at $p$, we can choose a local logarithm and thus a solution $f$ exists and $f(p) \neq 0$. However, if $p=p_i$ with $m_i=0$, the definition of $\mathcal{T}_C(-E)$ requires $f$ to have a zero at $p_i$, so again at these points $p_i$ there is no solution. Finally if $g$ has a pole at $p$, i.e. we have $p=p_i$ with $m_i<0$, a solution $f$ exists iff the order $m_i$ of the pole is divisible by $k$. This finishes the proof that the kernel $\mathcal{K}$ is the exceptional pushforward of a local system $\widetilde{\mathcal K}$ on the set $\tilde C$ above.
 
 To compute the cokernel, let $p \in U$ (which we identify with $0 \in \mathbb{C}$ below), let $h$ be a function on $U$, then we try to solve the differential equation
 \[\frac{f(z)}{g(z)} \frac{dg}{dz} + k \frac{df}{dz} = h(z)\]
 on $U$. For this we go to the branched covering $z=u^k$. Let $b=\text{ord}_p g$ and write $g(z) = z^b \tilde g(z)$. Let $g''$ be a local $k$-th rooth of $\tilde g$, then for $g'(u)=u^b g''(u^k)$ we have $g'(u)^k = g(u^k)$. With these preparations, our differential equation in coordinate $u$ has the form
 \[\frac{f(u^k)}{g'(u)^k} \frac{dg'}{du} g'(u)^{k-1} \frac{1}{u^{k-1}}  + \frac{df(u^k)}{du} \frac{1}{u^{k-1}} = h(u^k).\]
 Multiplying by $u^{k-1} g'(u)=u^{k+b-1} g''(u^k)$ we obtain
 \[\frac{d(f(u^k) g'(u))}{du} = u^{k+b-1} (g'' \cdot h) (u^k).\]
 The right hand side has a primitive $G$ iff its residue at $0$ vanishes. If this is the case, choose $G$ with constant term $0$ in its Laurent series around $u=0$, then the function $G(u)/g'(u)$ has only terms $u^{kj}, j \in \mathbb{Z}_{\geq 1}$ appearing in its power series, so we find a solution $f(z)$ vanishing at $0$. 
 
 The condition on the residue is automatic if $b$ is not negative and divisible by $k$ (that is $p \notin P$) and otherwise it is an obstruction with values in $\mathbb{C}$, so the cokernel of $\varphi$ has the form claimed above.
\end{proof}
In the case $k=0$, with notation as in the proof above, the function $\varphi(v)$ is locally given by $f \frac{d\log(g)}{dz}$. Thus if the partition $\mu$ is not the trivial partition $\mu=(0, \ldots, 0)$, the meromorphic function $g$ is nonconstant and hence $f \frac{d\log(g)}{dz}=0$ iff $f=0$. This shows that the kernel $\mathcal{K}$ is trivial. On the other hand, for a function $h$ the equation
\[f \frac{d\log(g)}{dz} = h\]
has the unique solution $f=h/\frac{d\log(g)}{dz}$ away from the zeroes of the derivative of $g$, so the cokernel $\mathcal{N}$ is again supported on a finite set in $C$.

We return now to the general case of $k \geq 0$. Let $\mathcal{I}=\text{im}(\varphi) \hookrightarrow \mathcal{O}_C$, then we have two short exact sequences of sheaves on $C$:
\begin{align*}
 0 \to \mathcal{K} \to \mathcal{T}_C(-E) \xrightarrow{\varphi} \mathcal{I} \to 0,\\
 0 \to \mathcal{I} \to \mathcal{O}_C \to \mathcal{N} \to 0.
\end{align*}
These induce long exact sequences in cohomology groups on $C$, given by
\begin{align*}
  \cdots \to H^1(\mathcal{K}) \to H^1(\mathcal{T}_C(-E)) \to H^1(\mathcal{I}) \to H^2(\mathcal{K}) \to 0,\\
  \cdots \to H^0(\mathcal{N}) \to H^1(\mathcal{I}) \to H^1(\mathcal{O}_C) \to 0.
\end{align*}
Here we use that the higher cohomologies of sheaves supported on isolated points vanish. Thus we have the following diagram
\begin{equation} \label{eqn:exseq}
\begin{tikzcd}
 & & & 0\\
 H^0(\mathcal{N}) \arrow{rd} & & H^2(\mathcal{K}) \arrow{ru} & \\
 & H^1(\mathcal{I}) \arrow{ru} \arrow{rd} & &\\
 H^1(\mathcal{T}_C(-E)) \arrow{rr}{H^1(\varphi)} \arrow{ru} & &H^1(\mathcal{O}_C) \arrow{rd} &\\
 & & & 0
\end{tikzcd}
\end{equation}
where the diagonal sequences are exact. From this we see immediately that $H^2(\mathcal{K})=0$ implies that $H^1(\varphi)$ is surjective, hence has trivial cokernel. This already finishes the proof of the case $k=0$. Fortunately, for $k \geq 1$ we can compute $H^2(\mathcal{K})$ easily.
\begin{Lem} \label{Lem:dimH2K}
 For $k \geq 1$ we have
 \[\dim\ H^2(\mathcal{K}) = \begin{cases}
                             1& \text{ if there is a merom. section $\tilde s_0$ of $\omega_C$ with $\tilde s_0^{\otimes k} = s_0$,}\\
                             0& \text{ otherwise}.
                            \end{cases}
\]
\end{Lem}
\begin{proof}
For the composition 
\[\tilde C \xrightarrow{\iota} C \xrightarrow{p} \{pt\},\]
the Leray spectral sequence gives us
\[R^ip_! R^j \iota_! \widetilde{\mathcal K} \implies R^{i+j}(p \circ \iota)_! \widetilde{\mathcal K}.\]
But note that the fibres of $\iota$ are either empty or single points and thus $R^j \iota_! \widetilde{\mathcal K} = 0$ for $j>0$. Recall now that for $f:X \to Y$ a continuous map of locally compact spaces and $\mathcal{F}$ a sheaf on $X$ we have $(R^i f_! \mathcal{F})_y = H_c^i(X_y,\mathcal{F}_y)$. Thus we can conclude
\[H^i(C, \iota_! \widetilde{\mathcal K}) = H^i_c(C, \iota_! \widetilde{\mathcal K}) = H^i_c(\tilde C, \widetilde{\mathcal K}).\]
Now for $i=2$ by \cite[Corollary 3.3.12]{MR2050072} we have an isomorphism
\[H^2_c(\tilde C, \widetilde{\mathcal K}) \cong H^0(\tilde C, \widetilde{\mathcal K}^\vee)^\vee.\]
Recall that $\widetilde{\mathcal K}$ parametrized tangent fields $v$ of $C$ with $\langle v^{\otimes k},s_0 \rangle = \text{const}$. From this we see that the dual local system $\widetilde{\mathcal K}^\vee$ parametrizes (meromorphic) sections $\tilde s_0$ of $\omega_C$ such that $\tilde s_0^{\otimes k}$ is a locally constant multiple of $s_0$. Such a section exists on all of $\tilde C$ iff we are in the first case of the Lemma above and then it is unique up to a constant, so indeed $H^2(C,\mathcal{K})$ is one-dimensional.
 %https://amathew.wordpress.com/2011/06/11/the-lower-shriek-and-base-change/
\end{proof}
The Lemma above is already sufficient to prove parts of Proposition \ref{Pro:dimTC}. Clearly condition (\ref{eqn:condhol}) implies on the one hand that $\dim\ H^2(\mathcal{K}) = 1$ and as there are no poles, we also have $\mathcal{N}=0$. Looking at the diagram (\ref{eqn:exseq}), we see $H^1(\mathcal{I}) \cong H^1(\mathcal{O}_C)$ and thus $\text{coker} H^1(\varphi) \cong H^2(\mathcal{K})$ is one-dimensional. This finishes the proof in this case.

In general, if there exists no meromorphic section $\tilde s_0$ of $\omega_C$ with $\tilde s_0^{\otimes k} = s_0$ our proof is also done, as in this case the cokernel of $H^1(\varphi)$ is trivial.

Thus we may from now on assume that we are in the strictly meromorphic case and that a section $\tilde s_0$ as above exists. Then from our description of $\widetilde{\mathcal K}^\vee$ it is obvious that it is the trivial local system $\underline{\mathbb{C}}$ on $\tilde C$ and thus the same is true for $\widetilde{\mathcal K}$. Moreover we must have that all weights $m_i$ in $\mu$ are divisible by $k$ (as $s_0$ has a $k$-th root), so there exists at least one marking $p_i$ with strictly negative weight $m_i$ divisible by $k$. 

With these assumptions, we claim that to finish the proof of Proposition \ref{Pro:dimTC}, it suffices to show the following Lemma.
\begin{Lem} \label{Lem:compononzero}
 The composition of morphisms
 \[H^0(\mathcal{N}) \xrightarrow{\delta} H^1(\mathcal{I}) \xrightarrow{\delta'} H^2(\mathcal{K})\]
 in the diagram (\ref{eqn:exseq}) above is nonzero and hence surjective.
\end{Lem}
Indeed this implies that while the map $H^1(\mathcal{T}_C(-E)) \to H^1(\mathcal{I})$ is not surjective, its image together with the kernel of $H^1(\mathcal{I}) \to H^1(\mathcal{O}_C)$ generate the space $H^1(\mathcal{I})$ and therefore $H^1(\varphi)$ is surjective.

\begin{proof}[Proof of Lemma \ref{Lem:compononzero}]
 Note that by Lemma \ref{Lem:dimH2K}, the space $H^2(\mathcal{K})$ is one-dimensional, so indeed the composition above is surjective iff it is nonzero. We will start with a suitable nonzero element in $H^0(\mathcal{N})$ and show that its image under $\delta' \circ \delta$ is nonzero.
 
 For this observe that $H^0(\mathcal{N}) = \bigoplus_{p \in P} \mathbb{C}$, i.e. we have a direct summand $\mathbb{C}$ for every $p=p_i$ with $m_i<0$ and divisible by $k$. We show that every one of these summands maps surjectively to $H^2(\mathcal{K})$. Let $p_i \in P$, then $\delta$ is a boundary map for the exact sequence
 \[0 \to \mathcal{I} \to \mathcal{O}_C \to \mathcal{N} \to 0.\]
 Cover $C$ by a small open disc $D$  around $p$ (such that $D \cap P = \{p_i\}$) and $U=C \setminus \{p\}$, then we can find a function $h \in \mathcal{O}_C(D)$ mapping to some nonzero $\lambda \in H^0(D,\mathcal{N}) = \mathbb{C}$. That is, there exists no solution $f$ of the equation $k \frac{df}{dz} + f \frac{d\log(g)}{dz} = h$ around $p$. Then the functions $h$ on $D$ and $0$ on $U$ map to the restrictions of the global section $\lambda \in H^0(D,\mathcal{N}) \subset H^0(C,\mathcal{N})$ to the cover $D,U$ of $C$.
 
 Thus the \v{C}ech $1$-cycle 
 \[(U \cap D, h|_{U \cap D}) \in H^1(\mathcal{I})\]
 is the image of $\lambda$ under $\delta$.
 
 For the map $\delta'$ coming from the exact sequence
 \[0 \to \mathcal{K} \to \mathcal{T}_C(-E) \xrightarrow{\varphi} \mathcal{I} \to 0\]
 we need to refine our cover of $C$: take still $D$ the open disc around $p_i$ and let $U^+ = C \setminus s^+$, $U^-=C \setminus s^-$, where $s^+, s^-$ are short path segments, starting in $p_i$ such that $D^+=D \cap U^+, D^-=D \cap U^-$ are simply connected (see also Figure \ref{Fig:cocycle}). Thus on $D^+,D^-$ we find vector fields $v^+=f^+ \frac{d}{d z}$, $v^-=f^- \frac{d}{d z}$ both mapping via $\varphi$ to the restrictions of $h$ to $D^+,D^-$. That is, the functions $f^\pm$ satisfy
 \begin{equation} \label{eqn:diffeq} k \frac{d f^\pm}{dz} + f^\pm \frac{d\log(g)}{dz} = h.\end{equation}
 This follows from the proof of Proposition \ref{Pro:kercoker} as $D^\pm$ are simply connected. Thus the difference $f^+-f^-$ on $D \cap U^+ \cap U^-$ solves the corresponding homogeneous equation describing the kernel $\mathcal{K}$ of $\varphi$ . By our assumptions, this kernel has a section $v_0 = f_0 \frac{d}{dz}$ on all of $D \setminus \{p_i\}$. Thus $f^+-f^-$ restricts to $a f_0, b f_0$ on the two components of $D \cap U^+ \cap U^-$ for some $a,b \in \mathbb{C}$. 
 
 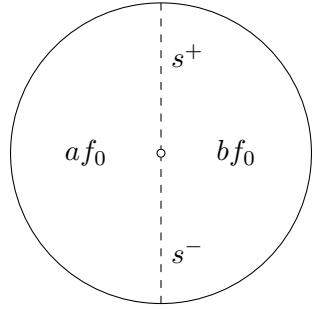
\begin{figure}[htb]
  \label{Fig:cocycle}
  \begin{tikzpicture}
   \draw (0,0) circle (2);
   \draw[dashed] (0,0) -- (0,1.3) node[right]{$s^+$} -- (0,2);
   \draw[dashed] (0,0) -- (0,-1.3) node[right]{$s^-$} -- (0,-2);
   \filldraw[white] (0,0) circle (1.5pt);
   \draw (0,0) circle (1.5 pt); %node[right]{$p_i$};
   \draw (-1,0) node{$a f_0$};
   \draw (1,0) node{$b f_0$};
  \end{tikzpicture}
  \caption{The restriction of $f^+-f^-$ to $D \cap U^+ \cap U^-$}
 \end{figure}
 Then $\delta'(\delta(\lambda))$ is given by $(D \cap U^+ \cap U^-, (f^+ - f^-) \frac{d}{dz}) \in H^2(C,\mathcal{D})$. 
 
 To show that this cycle class is nonzero, we first show that $a \neq b$. For this, we will go back to the computation of $\delta(\lambda)$ and adapt our choice of $h$. Consider the (multivalued) function $F(z)=f_0(z) \log(z)$ on $D$. It satisfies
 \begin{align*}
  k \frac{dF}{dz} + F(z) \log(g(z)) &= \frac{k f_0(z)}{z} + \underbrace{\left(k \frac{df_0}{dz} + f_0(z) \log(g(z)) \right)}_{=0} \log(z)\\ &= \frac{k f_0(z)}{z}=H.
 \end{align*}
 Then $H$ is single-valued and as $f_0$ (coming from $\mathcal{T}_C(-E)$) vanishes at $z=0$, it is holomorphic. Moreover, the solution space of the equation (\ref{eqn:diffeq}) with $h=H$ is given by $F + C f_0$ for $C \in \mathbb{C}$. Thus there exists no holomorphic solution around $p$ and $H$ maps to a nonzero element $\lambda' \in H^0(\mathcal{N})$. But then for $\delta'(\delta(\lambda'))$ we have $b-a=2 \pi i$, so in particular $a \neq b$. 
 
 Note that by assumption $\widetilde{\mathcal K} = \underline{\mathbb{C}}$ on $\tilde C$. Thus the injection $i_! \widetilde{\mathcal K} \to \underline{\mathbb{C}}$ induces a map of the second cohomology groups and the image of the two-cycle above in $H^2(C,\mathbb{C})$ is nonzero (see \cite[Chapter IV, Section 4.3]{scrap}). This finishes the proof.
\end{proof}

\section{Conjectural relation to Pixton's formula} \label{Sect:conjecture}
As in the appendix of \cite{fp} we want to state for any $k\geq 1$ a conjectural relation between a weighted fundamental class of $\widetilde{\mathcal H}_g^k(\mu)$ and an explicit element of the tautological ring $R(\overline{M}_{g,n})$ described by Pixton in \cite{pixtondr} and \cite{drcycles}.
\subsection{The weighted fundamental class of \texorpdfstring{$\widetilde{\mathcal H}_g^k(\mu)$}{tilde Hgk(mu)}} \label{Sect:Hgkmu}
As a first step, we identify the components $Z$ of $\widetilde{\mathcal H}_g^k(\mu)$ with dimension $2g-3+n$. Let $\Gamma$ be the dual graph of such a component, then the proof of \cite[Proposition 7]{fp} shows that 
\begin{enumerate}
 \item[(a)] there exists a unique center vertex $v_0$ in $\Gamma$ such that all edges of $\Gamma$ connect $v_0$ to one of the outlying vertices $v_1, \ldots, v_r$ or they are self-loops at $v_0$,
 \item[(b)] the vertices $v_1, \ldots, v_r$ can only have markings $p_i$ with nonnegative weigths $m_i$ divisible by $k$.
\end{enumerate}
Concerning the self-loops, Corollary \ref{Cor:irrsmooth} shows that for $\mathcal{M}_{g,n}^{\text{irr}} \subset \overline{\mathcal{M}}_{g,n}$ the locus of irreducible curves, we have
\[\overline{\mathcal{H}}_g^k(\mu) \cap \mathcal{M}_{g,n}^{\text{irr}} = \widetilde{\mathcal H}_g^k(\mu) \cap \mathcal{M}_{g,n}^{\text{irr}}.\]
Hence there is no component of $\widetilde{\mathcal H}_g^k(\mu)$ whose general element is an irreducible curve with at least one (self)node. This implies the additional restriction that
\begin{enumerate}
 \item[(c)] at the vertex $v_0$ of $\Gamma$ there is no self-loop. 
\end{enumerate}
For later use, recall that there is a gluing morphism
\[\xi_\Gamma : \prod_{v \in V(\Gamma)} \overline{\mathcal{M}}_{g(v),n(v)} \to \overline{\mathcal{M}}_{g,n},\]
where $g(v)$ is the genus of a vertex and $n(v)$ is the total number of markings and edges connected to the vertex $v$.

Let $S_{g,\mu}^k$ be the set of dual graphs satisfying the conditions (a), (b) and (c) above, called \emph{star graphs}. For $\Gamma \in S_{g,\mu}^k$ let $V_{\text{out}}(\Gamma)$ be the set of outlying vertices. A \emph{twist} for $\Gamma$ is a function 
\[I : E(\Gamma) \to k \mathbb{Z}_{>0}\]
such that
\begin{itemize}
 \item for the center vertex $v_0$ we have
 \[k(2 g(v_0) -2) + \sum_{e \mapsto v_0} (I(e)+k) = \sum_{i \mapsto v_0} m_i, \]
 \item for each $v_j \in V_{\text{out}}(\Gamma)$ we have
 \[k(2 g(v_j) -2) + \sum_{e \mapsto v_j} (-I(e)+k) = \sum_{i \mapsto v_j} m_i. \]
\end{itemize}
We denote by $\text{Tw}(\Gamma)$ the set of possible twists. For a vertex $v$ of $\Gamma$ let $\mu[v]$ be the vector of the weigths of markings $p_i$ mapping to $v$ and let $I[v]-k$ and $-I[v]-k$ be the vectors of numbers $I(e)-k$ and $-I(e)-k$ for edges $e$ adjacent to $v$.

Now if $Z \subset \widetilde{\mathcal H}_g^k(\mu)$ is a component of dimension $2g-3+n$ we must have that the generic dual graph $\Gamma$ is in $S_{g,\mu}^k$ and that there exists a twist $I \in \text{Tw}(\Gamma)$ such that $Z$ is a component of the closed set
\[\xi_\Gamma \left(\overline{\mathcal{H}}_{g}^k(\mu[v_0],-I[v_0] -k) \times \prod_{v \in V_{\text{out}}(\Gamma)} \overline{\mathcal{H}}_{g}^1\left(\frac{\mu[v]}{k},\frac{I[v]-k}{k}\right)\right).\]

We define a weighted fundamental class $H_{g,\mu}^k$ of $\widetilde{\mathcal H}_g^k(\mu)$ by the formula
\begin{align}
 H_{g,\mu}^k = \sum_{\Gamma \in S_{g,\mu}^k} \sum_{I \in \text{Tw}(\Gamma)} \frac{\prod_{e \in E(\Gamma)} I(e)}{|\text{Aut}(\Gamma)| \cdot k^{|V_{\text{out}}(\Gamma)|}} (\xi_\Gamma)_* \big[\left[\overline{\mathcal{H}}_{g(v_0)}^k(\mu[v_0],-I[v_0] -k)\right]\nonumber \\\cdot \prod_{v \in V_{\text{out}}(\Gamma)} \left[\overline{\mathcal{H}}_{g(v)}^1\left(\frac{\mu[v]}{k},\frac{I[v]-k}{k}\right)\right]\big] \label{eqn:Hgmukformula}
\end{align}

\subsection{Pixton's tautological class}
Consider now the shifted vector
\[\tilde \mu = (\tilde m_1, \ldots, \tilde m_n), \tilde m_i = m_i + k.\]
Then in \cite[Section 1.1]{drcycles} a tautological class
\[P_{g,\mu}^{g,k} = P_g^{g,k}(\tilde \mu) \in R^g(\overline{\mathcal{M}}_{g,n})\]
is defined. We make the following conjecture, extending the case $k=1$ from \cite{fp}.

\noindent \textbf{Conjecture A.} For $k \geq 1$ and $\mu$ not of the form $\mu = k \mu'$ for a nonnegative partition $\mu'$ of $2g-2$, we have \[H_{g,\mu}^k = 2^{-g} P_{g,\mu}^{g,k}.\]

With Theorem \ref{Theo:dim} we see that the condition on $\mu$ above exactly ensures that all components of $\widetilde{\mathcal H}_g^k(\mu)$ are of codimension $g$. However, we also want to propose a conjecture in the case $\mu = k \mu'$ with $\mu'$ holomorphic. Here we have noted that the codimension $g-1$ part of $\widetilde{\mathcal H}_g^k(\mu)$ is exactly given by
\[\overline{\mathcal{H}}_g^1(\mu') \subset \widetilde{\mathcal H}_g^k(\mu).\]
Now the idea is to use the formula of Conjecture A for $k=1$ and $\mu'$ to assign to this locus a virtual fundamental class of codimension $g$. Denote by $\text{Cont}_{g,\mu}^k(\Gamma, I)$ the contribution of graph $\Gamma$ and twist $I$ to $H_{g,\mu}^k$ in (\ref{eqn:Hgmukformula}). Then this virtual fundamental class is defined by the formula
\begin{align*}
 \left[\overline{\mathcal{H}}_g^1(\mu') \right]^{\text{vir}} + \sum_{\substack{\Gamma \in S_{g,\mu'}^1\\\Gamma\text{ nontrivial}}} \sum_{I \in \text{Tw}(\Gamma)} \text{Cont}_{g,\mu'}^1(\Gamma, I) = 2^{-g} P_{g,\mu'}^{g,1}.
\end{align*}
The conjecture we want to propose is that when we write down the formula of Conjecture A for $k$ and $\mu=k \mu'$ and replace the codimension $g-1$ part $\left[\overline{\mathcal{H}}_g^1(\mu') \right]$ of $\left[\overline{\mathcal{H}}_g^k(\mu) \right]$ by $\left[\overline{\mathcal{H}}_g^1(\mu') \right]^{\text{vir}}$, we obtain a true equality of codimension $g$ cycle classes. To make this precise, let
\[\mathcal{H}_g^k(\mu)' = \mathcal{H}_g^k(\mu) \setminus \mathcal{H}_g^1(\mu').\]
Then we propose the following.

\noindent \textbf{Conjecture A$'$.} For $k \geq 1$ and $\mu=k\mu'$ for a nonnegative partition $\mu'$ of $2g-2$, we have \[\left[\overline{\mathcal{H}}_g^1(\mu') \right]^{\text{vir}} + \left[ \overline{\mathcal{H}}_g^k(\mu)'\right] + \sum_{\substack{\Gamma \in S_{g,\mu}^k\\\Gamma\text{ nontrivial}}} \sum_{I \in \text{Tw}(\Gamma)} \text{Cont}_{g,\mu}^k(\Gamma, I)= 2^{-g} P_{g,\mu}^{g,k}.\]

% Note that we exclude the cases $g=0,1$ above as here there are either no nonnegative partitions of $2g-2$ or only the trivial partition.

\subsection{Examples}
For a list of examples where the case $k=1$ of Conjecture A has been verified see \cite[Section A.5]{fp}.
\subsubsection{Genus at most 1}
For $g=0$, Conjecture A is trivial and true for all $k \geq 1$ and Conjecture A$'$ is empty, as there is no nonnegative partition of $-2$.

For $g=1$ and a partition $\mu$ of $k(2g-2)=0$ we can first look at $H_{1,\mu}^k$. As expected we have a contribution of $[\overline{\mathcal{H}}_{1}^k(\mu)]$ from the trivial star graph. But $\mathcal{H}_{1}^k(\mu)$ parametrizes points $(C,p) \in \overline{\mathcal{M}}_{1,n}$ with $\mathcal{O}_C(\sum_{i} m_i p_i) = \mathcal{O}_C$, so it is independent of $k$ and we have
\[[\overline{\mathcal{H}}_{1}^k(\mu)] = [\overline{\mathcal{H}}_{1}^1(\mu)].\]
As outlying vertices must have genus at least $1$, we only have one more class of star graphs $\Gamma$ contributing to $H_{1,\mu}^k$, namely those with 
\begin{itemize}
 \item exactly one edge $e$ between a genus $0$ vertex $v_0$ and a genus $1$ vertex $v_1$,
 \item all markings going to $v_1$ having weight $0$,
 \item the unique twist $I(e)=k$.
\end{itemize}
For $I \subset \{1, \ldots, n\}$ with $|I|\leq n-2$ denote by $\delta_{I}$ the divisor in $\overline{\mathcal{M}}_{1,n}$ of curves with two components of genera $0$ and $1$ where the genus $1$ component carries the markings in $I$. Then we have
\[H_{1,\mu}^k = [\overline{\mathcal{H}}_{1}^1(\mu)] + \sum_{I \subset \{i:m_i=0\}, |I| \leq n-2} \delta_I.\]
In particular this is independent of $k$, so $H_{1,\mu}^k = H_{1,\mu}^1$. 

On the other hand we compute directly from the definition in \cite{drcycles} that
\[P_{1,\mu}^{1,k} = -k^2 \kappa_1 + \sum_{i=1}^n (m_i+k)^2 \psi_i - \frac{1}{12} \delta_{\text{irr}} - \sum_{|I| \leq n-2} (k-\sum_{i \in I} m_i)^2 \delta_I.\]
Here $\kappa_1=\pi_*(\psi_{n+1}^2)$ is the first kappa-class (with $\pi: \overline{\mathcal{M}}_{g,n+1} \to \overline{\mathcal{M}}_{g,n}$ the forgetful map) and 
\[\delta_{\text{irr}} = \frac{1}{2} \xi_* \left[ \overline{\mathcal{M}}_{0,n+2} \right], \quad \xi:\overline {\mathcal{M}}_{0,n+2} \to \overline{\mathcal{M}}_{0,n}\]
is the divisor class of curves with a non-separating node in $\overline{\mathcal{M}}_{1,n}$. %delta_irr with isomorphisms divided yet

At first this seems to depend on $k$. However, using the relations
\[\kappa_1=\sum_{i=1}^n \psi_i - \sum_{|I| \leq n-2} \delta_I, \psi_i = \frac{1}{12} \delta_{\text{irr}} + \sum_{I \not \ni i, |I|\leq n-2} \delta_I\]
a straightforward computation shows that
\[P_{1,\mu}^{1,k} = \sum_{i=1}^n m_i^2 \psi_i - \frac{1}{6} \delta_{\text{irr}} - \sum_{|I| \leq n-2} (\sum_{i \in I} m_i)^2 \delta_I.\]
This expression is now also independent of $k$ so $P_{1,\mu}^{1,k} = P_{1,\mu}^{1,1}$. Hence Conjecture A for $k\geq 1$ follows from the case $k=1$, which was already shown in \cite{fp}. On the other hand Conjecture A$'$ is also true, as all terms for $k=1$ and $\mu'=(0, \ldots, 0)$ exactly cancel the terms for $k$ and $\mu$.

\subsubsection{Genus 2} \label{Sect:ExaGen2}
For genus $g=2$ and $k=2$ below we check Conjectures A and A$'$ in two cases each. Here the first such test is presented in considerable detail, listing all star graphs and their contributions. For the remaining three, we only give indications how the classes $[\overline{\mathcal{H}}_g^k(\mu)], [\overline{\mathcal{H}}_g^k(\mu)']$ can be determined, as the remaining contributions are straightforward to compute.

As a first example, we look at the partition $\mu=(3,1)$. What makes this case easy to check is that $\overline{\mathcal{H}}_2^2(3,1) = \emptyset$ (see \cite[Theorem 2]{masursmillie}), so the trivial graph does not give a contribution to $H_{2,(3,1)}^2$. The terms coming from nontrivial graphs are listed below, where vertices are labelled with their genus and edges with their twist. Here the central vertex is always the vertex on the left.

\vspace{5 pt}
\begin{tabular}{cl}
 Graph $\Gamma$ & Contribution\\
 \raisebox{-0.5\height}{\begin{tikzpicture}
  \filldraw (0,0) circle (2pt) node[above]{\tiny{\textbf{1}}};
  \draw (0,0) -- (-0.5,0.15) node[left]{\tiny{$3$}};
  \draw (0,0) -- (-0.5,-0.15) node[left]{\tiny{$1$}};
  \draw (0,0) -- (1,0) node[above]{\tiny{$2$}} -- (2,0);
  \filldraw (2,0) circle (2pt) node[above]{\tiny{\textbf{1}}};  
 \end{tikzpicture}}
 & $(\xi_\Gamma)_* \left[\left[\overline{\mathcal{H}}_{1}^2(3,1,-4)\right] \cdot  \left[\overline{\mathcal{M}}_{1,1}\right]\right]$\\
 \raisebox{-0.5\height}{\begin{tikzpicture}
  \filldraw (0,0) circle (2pt) node[above]{\tiny{\textbf{0}}};
  \draw (0,0) -- (-0.5,0.15) node[left]{\tiny{$3$}};
  \draw (0,0) -- (-0.5,-0.15) node[left]{\tiny{$1$}};
  \draw (0,0) -- (1,0) node[above]{\tiny{$6$}} -- (2,0);
  \filldraw (2,0) circle (2pt) node[above]{\tiny{\textbf{2}}};  
 \end{tikzpicture}}
 & 3 $(\xi_\Gamma)_* \left[\left[\overline{\mathcal{M}}_{0,3}\right] \cdot  \left[\overline{\mathcal{H}}_{2}^1(2)\right]\right]$\\
 \raisebox{-0.5\height}{\begin{tikzpicture}
  \filldraw (0,0) circle (2pt) node[above]{\tiny{\textbf{0}}};
  \draw (0,0) -- (-0.5,0.15) node[left]{\tiny{$3$}};
  \draw (0,0) -- (-0.5,-0.15) node[left]{\tiny{$1$}};
  \draw (0,0) -- (1,0.4) node[above]{\tiny{$2$}} -- (2,0.8);
  \filldraw (2,0.8) circle (2pt) node[above]{\tiny{\textbf{1}}};  
  \draw (0,0) -- (1,-0.4) node[above]{\tiny{$2$}} -- (2,-0.8);
  \filldraw (2,-0.8) circle (2pt) node[above]{\tiny{\textbf{1}}};  
 \end{tikzpicture}}
 & $\frac{1}{2} (\xi_\Gamma)_* \left[\left[\overline{\mathcal{M}}_{0,4}\right] \cdot  \left[\overline{\mathcal{M}}_{1,1}\right] \cdot \left[\overline{\mathcal{M}}_{1,1}\right] \right]$\\
 \raisebox{-0.5\height}{\begin{tikzpicture}
  \filldraw (0,0) circle (2pt) node[above]{\tiny{\textbf{0}}};
  \draw (0,0) -- (-0.5,0.15) node[left]{\tiny{$3$}};
  \draw (0,0) -- (-0.5,-0.15) node[left]{\tiny{$1$}};
  \draw (0,0) to [bend left=20] (2,0);
  \draw (0,0) to [bend right=20] (2,0);
  \draw (1,0.2) node[above]{\tiny{$2$}};
  \draw (1,-0.2) node[below]{\tiny{$2$}};
  \filldraw (2,0) circle (2pt) node[above]{\tiny{\textbf{1}}};  
 \end{tikzpicture}}
 & $(\xi_\Gamma)_* \left[\left[\overline{\mathcal{M}}_{0,4}\right] \cdot  \left[\overline{\mathcal{M}}_{1,2}\right]\right]$
\end{tabular}

Here the class $[\overline{\mathcal{H}}_{1}^2(3,1,-4)]$ is obtained from the case $g=1$ of Conjecture A verified above whereas the class $[\overline{\mathcal{H}}_{2}^1(2)]$ has been computed in \cite{eisenbudharris}. Summing up the contributions above we obtain a tautological cycle which exactly equals $2^{-2} P_{2,(3,1)}^{2,2}$. To verify this equality one uses known relations in the tautological ring. 

Another check for Conjecture A is possible for $g=2,k=2$ and $\mu=(2,1,1)$. Here, by \cite[Theorem 1.2]{lanneau} (or an elementary argument involving the Residue theorem), we have
\[\mathcal{H}_2^2(2,1,1) = \left\{(C,q,p_1, p_2) : \parbox{0.4\linewidth}{$q$ Weierstrass point, \\ $p_1, p_2$  hyperelliptic conjugate}\right\}.\]
The class of the closure of this locus has been computed in \cite[Lemma 3]{belopandh}. As for the other terms appearing in $H_{2,(2,1,1)}^2$, all of them are obtained from the genus $1$ case of Conjecture A or in \cite{belopandh}. Again the statement of Conjecture A for $\mu=(2,1,1)$ is true.

Concerning Conjecture A$'$, we are able to verify it in two cases, namely $g=2, k=2$ and $\mu=(4)$ or $\mu=(2,2)$. For the partition $\mu=(4)$, it follows from \cite[Theorem 1.2]{lanneau} that we have
\[\mathcal{H}_2^2(4) = \mathcal{H}_2^1(2),\]
so $\mathcal{H}_2^2(4)'=0$. All contributions from nontrivial graphs are easily computed and the conjecture holds in this case.

For $\mu=(2,2)$ we find that $(C,p,q) \in \mathcal{H}_2^2(2,2)$ iff either $p,q$ are hyperelliptic conjugate (i.e. $(C,p,q) \in \mathcal{H}_2^1(1,1)$) or if both $p,q$ are Weierstrass points. Thus we have
\[\mathcal{H}_2^2(2,2)' = \{(C,p,q): p,q \text{ Weierstrass points}\}\]
and the class of the closure of this locus has been computed by Tarasca in \cite[Theorem 0.1]{tarasca} (under the name $[\overline{\mathcal{DR}}_2(2)]$). Again all other terms are computed via the genus $1$ case of Conjecture A and the claimed relation holds.

All four equalities in this section were checked by Aaron Pixton, using a computer program to expand the cycles $2^{-g} P_{g,\mu}^{g,k}$ in terms of tautological classes and then using known relations (\cite{ppz}) in the tautological ring to show the equality to $H_{g,\mu}^k$. We gratefully acknowledge his help.

\subsection{Recursions for \texorpdfstring{$\overline{\mathcal{H}}_g^k(\mu)$}{bar Hgk(mu)}}
Let $k \geq 1$ and let $\mu$ be a partition of $k(2g-2)$ not of the form $\mu=k \mu'$ for a nonnegative partition $\mu'$ of $2g-2$. Assuming Conjecture A above, we can determine an expression for $\overline{\mathcal{H}}_g^k(\mu)$ in terms of tautological classes. Indeed, in the formula for $H_{g,\mu}^k$, the term $[\overline{\mathcal{H}}_g^k(\mu)]$ appears with a coefficient $1$. All other terms are composed from cycles
\begin{itemize}
 \item $[\overline{\mathcal{H}}_{g'}^k(\mu')]$ for $g'<g$ (as outlying vertices must have genus at least $1$),
 \item $[\overline{\mathcal{H}}_{g''}^1(\mu'')]$, which were determined in \cite{fp}.
\end{itemize}
Hence in the equation $H_{g,\mu}^k = 2^{-g} P_{g,\mu}^{g,k}$, we can solve for $[\overline{\mathcal{H}}_g^k(\mu)]$ and determine these cycles by induction on $g$.

\begin{appendix}
\section{Smoothing of untwisted nodes for \texorpdfstring{$k$}{k}-differentials}
In this section, for the convenience of the reader, we give an elementary argument showing that an untwisted node in a curve $(C,p_1, \ldots, p_n) \in \widetilde{\mathcal H}_{g}^k(\mu)$ can be smoothed in a holomorphic $1$-parameter family inside $\widetilde{\mathcal H}_{g}^k(\mu)$. The proof below was explained to us by Dimitri Zvonkine. 

Let $\mathbb{D} =  \{z : |z| < 1\} \subset \mathbb{C}$ be the unit disc.
\begin{Pro}
 Let $g,n,k \geq 0$ with $2g-2+n>0$ and let $\mu$ be a partition of $k(2g-2)$. Assume we have  \[(C,p_1, \ldots, p_n) \in \widetilde{\mathcal H}_{g}^k(\mu)\subset \overline{\mathcal{M}}_{g,n}\]
 where $C$ has a node $q$ at the intersection of two components $D,D'$ (where possibly $D=D'$) such that there exists a twist $I$ for the divisor $\sum_{i} m_i p_i$ with $I(q,D)=I(q,D')=0$. 
 
 Then there exists a holomorphic $1$-parameter family $\pi: \mathcal{C} \to \mathbb{D}$ and sections $\sigma_1, \ldots, \sigma_n : \mathbb{D} \to \mathcal{C}$ with central fibre 
 \[(\mathcal{C}_0, \sigma_1(0), \ldots, \sigma_n(0)) \cong (C,p_1, \ldots, p_n),\]
 smoothing the node $Q$ such that the induced analytic map $\varphi: \mathbb{D} \to \overline{\mathcal{M}}_{g,n}$ has image in $\widetilde{\mathcal H}_{g}^k(\mu)$.
\end{Pro}
\begin{proof}
 Let $U \subset C$ be an analytic neighbourhood of $Q$ with an isomorphism $\psi: U \xrightarrow{\sim} V(xy) \subset \mathbb{C}^2$ and not containing any of the points $p_1, \ldots, p_n$. Let 
 \[V=C \setminus \psi^{-1}(V(xy) \cap \{(x,y): |x| \leq 1, |y| \leq 1\}).\]
 Then $U,V$ form an open cover of $C$ and their intersection is a union of two annuli
 \[U \cap V \cong \underbrace{\{(x,0) : |x|>1\}}_{=:A_x} \amalg \underbrace{\{(0,y) : |y|>1\}}_{=:A_y}.\]
 We construct the family $\pi: \mathcal{C} \to \mathbb{D}$ by gluing the trivial family $\pi_V: V \times \mathbb{D} \to \mathbb{D}$ with constant sections $\sigma_i(t)=(p_i,t)$ to a family $\pi_U: \mathcal{U} \to \mathbb{D}$, which smoothes the node of the central fibre $\mathcal{U}_0 = U$. The gluing happens along the subset $(A_x \amalg A_y) \times \mathbb{D}$ of $V \times \mathbb{D}$.
 \begin{center}
  \begin{tikzcd}
    \mathcal{U} \arrow[hookleftarrow]{r}{i_{\mathcal{U}}} \arrow{d}& (A_x \amalg A_y) \times \mathbb{D} \arrow{d} \arrow[hookrightarrow]{r} & V \times \mathbb{D}\arrow{d}\\
   \mathbb{D} \arrow[-, double equal sign distance]{r} & \mathbb{D} \arrow[-, double equal sign distance]{r} &\mathbb{D} %\mathbb{D} \arrow[equal]{r} & \mathbb{D} \arrow[equal]{r} &\mathbb{D}
  \end{tikzcd}
 \end{center}
 After choosing suitable coordinates on $U$ (i.e. modifiying the isomorphism $\psi: U \to V(xy)$), the family $\mathcal{U}$ is easy to write down:
 \[\mathcal{U} = \{(x,y,t) \in \mathbb{C}^2 \times \mathbb{D} : xy-t=0\} \xrightarrow{(x,y,t) \mapsto t} \mathbb{D}\]
 and the map $i_{\mathcal{U}}$ sends $(x,t) \in A_x \times \mathbb{D}$ to $(x,t/x,t)$ and $(y,t) \in A_y \times \mathbb{D}$ to $(t/y,y,t)$. 
 
 It is easy to check that the above data defines a holomorphic family $\pi$ smoothing the node $q$ of the central fibre $(C,p_1, \ldots, p_n)$. The crucial point however is to ensure that the corresponding twisted differential deforms in this family, i.e. that the map $\mathbb{D} \to  \overline{\mathcal{M}}_{g,n}$ has image in $\widetilde{\mathcal H}_{g}^k(\mu)$.
 
 By assumption, for the partial normalization $\nu : C_I \to C$ resolving all nodes $p$ of $C$ with nonzero twist $I$ we have a section $\eta$ of $\nu^* \omega_C$ having zeroes and poles at the points $\nu^{-1}(p_i)$ with order $m_i$ and at the preimages of $p',p''$ of the nodes $p$ with orders determined by the twist $I$. As the twist $I$ vanishes at the node $q$, it is not normalized by $\eta$, so over $U \subset C$ the map $\eta$ is an isomorphism. Using this, we obtain a section $\eta_U \in \Gamma(U, \omega_U^{\otimes k})$. 
 
 On the horizontal and vertical components $V(y), V(x)$ of $U$, we can express the restrictions $\eta_x, \eta_y$ of $\eta_U$ in terms of the coordinate $x$ on $V(y)$ and $y$ on $V(x)$ as
 \begin{align*}
  \eta_x = a &(1 + b_{k-1}x + b_{k-2} x^2 + \ldots) \left(\frac{dx}{x} \right)^k,\\
  \eta_y = (-1)^k a &(1 + c_{k-1}y + c_{k-2} y^2 + \ldots) \left(\frac{dy}{y} \right)^k.
 \end{align*}
 Here the number $a \neq 0$ is invariant under coordinate changes. The fact that the order $(-k)$-coefficients of $\eta_x, \eta_y$ agree (up to a sign $(-1)^k$) follows from the usual fact that for $k=1$ the residues of the differentials on both components have to sum up to zero. Then for any local section $s$ of $\omega_C$ not vanishing around $q$, $s^k$ is a local section of $\omega^k$ around $q$ and all such sections are multiples of $s^k$. 
 
 In the case $k \geq 1$, we claim that after a holomorphic change of coordinates $x=x(w)$ with $x(0)=0$ we can assume that $\eta_x$ is of the form $\eta_x = a \left(\frac{dw}{w} \right)^k$. 
 Indeed, for $f(x)=1 + b_{k-1}x + b_{k-2} x^2 + \ldots$ as above choose a $k$-th root $g(x)$ of $f$ around $x=0$ with $g(0)=1$. 
 Then for a change of coordinates $x(w)$ we compute
 \begin{align*}
  \eta_x &= a f(x(w)) \left(\frac{d x(w)}{x(w)}\right)^k = a f(x(w)) \left(\frac{d x(w)}{dw} \frac{w}{x(w)}\right)^k \left(\frac{dw}{w} \right)^k.
 \end{align*}
 Thus we want to find a solution of the differential equation
 \begin{equation} \label{eqn:xweqn} f(x(w)) \left(\frac{d x(w)}{dw} \frac{w}{x(w)}\right)^k = 1.\end{equation}
 Dividing by $f(x(w))$ and taking a $k$th rooth, the equation is in particular satisfied if we have
 \[\frac{d x(w)}{dw} \frac{w}{x(w)} = g(x(w))^{-1}.\]
 We solve this equation by separation of variables. Define $R(z)=g(z)/z$ observing $R(z)=1/z + \tilde R(z)$ with $\tilde R$ holomorphic around $z=0$. Then we write the equation above as
 \[\frac{d x(w)}{dw} R(x(w)) = w^{-1}.\]
  Define the (multivalued) function
 \[S(u) = \int_{z_0}^u R(z) dz = (\log(u)-\log(z_0) ) + \underbrace{\int_{z_0}^u \tilde R(z) dz}_{=:\tilde S(u)}\]
 for a fixed $z_0$ close to $0$.
 Note that $\tilde S$ is well-defined, only the logarithm terms are only unique up to $2 \pi i \mathbb{Z}$. Then the above equation has the form
 \[\frac{d}{dw} S(x(w)) = w^{-1}.\]
 Integrating and then taking the exponential on both sides, we obtain
 \[\frac{x(w)}{z_0}\exp(\tilde S(x(w))) = Aw\]
 for some $A \in \mathbb{C}^*$. But the function $u \mapsto u \exp(\tilde S(u))$ has nonzero derivative at $u=0$, so we can locally find a holomorphic inverse function $T$. We conclude
 \[x(w)= T(A z_0 w).\]
 Going backwards through the argument above one checks that this choice of $x(w)$ then satisfies the original equation (\ref{eqn:xweqn}).

 Thus after a change of coordinates and possibly shrinking $U$, we can assume that $\eta_x, \eta_y$ are given by
 \[\eta_x = a \left( \frac{dx}{x}\right)^k, \eta_y = (-1)^k a \left( \frac{dy}{y}\right)^k.\]
 On the fibre $\mathcal{U}_t = V(xy-t) \subset \mathbb{C}^2$ for $t$ fixed we have $0=d(xy-t) = x dy + y dx$, so $dx/x = - dy/y$ if $x\neq 0, y \neq 0$. Using this, we see that the meromorphic differentials $a \left( \frac{dx}{x}\right)^k, (-1)^k a \left( \frac{dy}{y}\right)^k$ on $\mathbb{C}^2 \times \mathbb{D}$ restrict to the same relative $k$-differential on $\mathcal{U} \setminus \{(0,0,0)\} \to \mathbb{D}$. As $\mathcal{U}$ is isomorphic to an open subset of $\mathbb{C}^2$ and as $(0,0,0)$ has codimension $2$, this extends to a unique relative differential $\eta_{\mathcal{U}}$ extending $\eta_{U}$ on the central fibre by the second Riemann extension theorem. 
 
 Note that in the case $k=0$, which was omitted so far, the argument is even easier: for $\eta_x, \eta_y$ the restrictions of $\eta_U$ to the $x$ and $y$-axis as above, we note that the function $(x,y,t) \mapsto \eta_x(x) + \eta_y(y) - a$ on $\mathbb{C}^2 \times \mathbb{D}$ restricts to a function $\eta_{\mathcal{U}}$ on $\mathcal{U}$ extending $\eta_U$ on the central fibre. 
 
 We can now verify that the pullback of $\eta_{\mathcal{U}}$ by $i_{\mathcal{U}} : (A_x \amalg A_y) \times \mathbb{D} \to \mathcal{U}$ is independent of the point $t \in \mathbb{D}$. This shows that we can glue this $k$-differential in the family $\mathcal{C}$ and that for all $t \in \mathbb{D}$ we obtain a differential as desired on the partial normalization of the fibres $\mathcal{C}_t$. But for instance for the restriction of $i_{\mathcal{U}}$ to $A_x \times \mathbb{D}$ given by $(x,t) \mapsto (x,t/x,t)$ we have that 
 \[i_{\mathcal{U}}^* \eta_{\mathcal{U}} = i_{\mathcal{U}}^* a \left( \frac{dx}{x}\right)^k = a \left( \frac{dx}{x}\right)^k\]
 is indeed independent of $t$ (and similarly for $A_y \times \mathbb{D}$). Again the case $k=0$ is even more obvious.
 
%  Then we set 
%  \[x(w) = \exp\left( \int_{z_0}^w g(x)^{-1} \frac{dx}{x} \right)\]
%  for some fixed $z_0$ close to $0$. To see that this gives a well-defined holomorphic function in a neighbourhood of $w=0$ with nonzero derivative at $0$, write $g(x)^{-1} = 1 + h(x)$ for $h(0)=0$. Then we have
%  \begin{align*}
%   \exp\left( \int_{z_0}^w g(x)^{-1} \frac{dx}{x} \right) &= \exp\left( \int_{z_0}^w  \frac{dx}{x} \right) \exp\left( \int_{z_0}^w \underbrace{\frac{h(x)}{x}}_{\text{hol. around }x=0} dx \right)\\
%   &=\frac{w}{z_0} \exp\left( \int_{z_0}^w \frac{h(x)}{x} dx \right).
%  \end{align*}
%  We compute
%  \begin{align*}
%   \frac{dx}{x} &= \frac{d \log(x(w))}{dw} dw = 
%  \end{align*}
%  observe that
%  \[\int_{\partial B_{\epsilon}(0)} g(x)^{-1} \frac{dx}{x} = 2 \pi i g(0)^{-1} = 2 \pi i,\]
%  so after applying $\exp$ the value of $x(w)$ is independent of the chosen path from $z_0$ to $w$.
% 
%  Moreover, we construct a section $\eta$ of $\omega_{\pi_U}^{\otimes k}$
\end{proof}

\begin{Cor} \label{Cor:irrsmooth}
 Let $g,n,k \geq 0$ with $2g-2+n>0$, $\mu$ a partition of $k(2g-2)$ and $(C,p_1, \ldots, p_n) \in \widetilde{\mathcal H}_{g}^k(\mu)$. Then if $C$ is irreducible, we have $(C,p_1, \ldots, p_n) \in \overline{\mathcal H}_{g}^k(\mu)$.
\end{Cor}
\begin{proof}
 We do an induction on the number $\delta$ of nodes of $C$, which are always self-nodes, hence untwisted. The case $\delta=0$ is clear and the case $\delta \geq 1$ follows from the statement for $\delta-1$ as we can smooth $C$ in a family with general element having $\delta-1$ nodes. 
\end{proof}
The case $k=1$ and $\mu$ holomorphic of this result was proved in \cite[Lemma 12]{fp}.

\end{appendix}

\bibliographystyle{alpha}
\bibliography{Biblio}

\end{document}